\documentclass[a4paper]{amsart}
\usepackage{times}
\usepackage{cite}
\usepackage{amsthm,amsmath,amssymb,amsfonts}
\usepackage{algorithmic}
\usepackage{graphicx}
\usepackage{textcomp}
\usepackage{mathtools}

\usepackage{bussproofs}

\usepackage{relsize}

\usepackage{amsmath}
\usepackage{placeins}
\usepackage{array}

\usepackage{hyperref}
\usepackage{bussproofs}

\newcommand{\RULE}[2]{\cfrac{#1}{\raisebox{-1mm}{\ensuremath{#2}}}}

\usepackage{tikz}
\usetikzlibrary{arrows,chains,matrix,positioning,scopes}
\def\BibTeX{{\rm B\kern-.05em{\sc i\kern-.025em b}\kern-.08em
    T\kern-.1667em\lower.7ex\hbox{E}\kern-.125emX}}

\makeatletter
\tikzset{join/.code=\tikzset{after node path={%
\ifx\tikzchainprevious\pgfutil@empty\else(\tikzchainprevious)%
edge[every join]#1(\tikzchaincurrent)\fi}}}
\makeatother

\tikzset{>=stealth',every on chain/.append style={join},
         every join/.style={->}}
\tikzstyle{labeled}=[execute at begin node=$\scriptstyle,
   execute at end node=$]

\newtheorem{Thm}{Theorem}
\newtheorem{Rmk}[Thm]{Remark}

\newtheorem{Pro}[Thm]{Proposition}

\usepackage{color,dsfont}

\theoremstyle{plain}

\theoremstyle{definition}
\usepackage[]{graphicx}

\title{\bf Axiomatization via translation: Hi{\. z}'s warning \\ for predicate logic}

 \author{Guillermo Badia }
\address{University of Queensland, Brisbane, Australia}
\email{g.badia@uq.edu.au}
 \author{John N. Crossley}
 \address{Monash University, Melbourne, Australia}
 \email{john.crossley@monash.edu}
   \author{Lloyd Humberstone }
\address{Monash University, Melbourne, Australia}
 \email{lloyd.humberstone@monash.edu}
 
 \usepackage{fancyhdr}

\begin{document}


\maketitle

\begin{abstract}
The problems of logical translation of axiomatizations and the choice of primitive operators have surfaced several times over the years. An early issue was raised by  H. Hi{\. z} in the 1950s on the  incompleteness of translated calculi. Further pertinent work, some of it touched on here, was done in the 1970s by W. Frank and S. Shapiro, as well as by others in subsequent decades.  As we shall see, overlooking such possibilities has led to incorrect claims  of completeness being made  (e.g. by  J. L. Bell and A. B. Slomson as well as J. N. Crossley) for axiomatizations of classical predicate logic obtained by translation from axiomatizations suited to differently chosen logical primitives.  In this note we begin by discussing some problematic aspects of an early article by W. Frank on the difficulties of obtaining completeness theorems for translated calculi. Shapiro had established the incompleteness of Crossley's axiomatization by exhibiting a propositional tautology that was not provable.   In contrast, to deal with Bell and Slomson's system which is complete for propositional tautologies, we go on to show that taking a formal system for classical predicate calculus with the primitive $ \exists$, setting  $\forall x \phi(x) \stackrel{\text{def}}{=}\neg \exists x \neg \phi(x)$, and writing down a set of axioms and rules complete for the calculus with $\forall $ instead of $ \exists$ as primitive, does not guarantee completeness of the resulting system. In particular, instances of the valid schema $\exists x \phi (x) \rightarrow \exists x \neg \neg\phi (x)$ are not provable, which is analogous to what occurs in modal logic with $\Box$ and $\Diamond$.

\medskip

 \medskip

 \end{abstract}
\section{Introduction} \label{Intro}

A translation \textbf{t} from a language $\mathcal{L}$, thought of as the set of its formulas, to the language $\mathcal{L}'$ is a function from $\mathcal{L}$ to $\mathcal{L}'$ (on which one may impose further demands of compositionality etc., as desired). We say that \textbf{t} \textit{embeds} a consequence relation $\vdash$ on $\mathcal{L}$ in a consequence relation $\vdash'$ on $\mathcal{L}'$ when for all $\phi_1,\ldots, \phi_n, \psi \in \mathcal{L}$ when $\phi_1,\ldots, \phi_n \vdash \psi$ only if ${\bf t}(\phi_1),\ldots, {\bf t}(\phi_n) \vdash' {\bf t}(\psi)$ and that {\bf t} does so \emph{faithfully} when we have this with `only if' strengthened to `if and only if'.\footnote{A taxonomy of translations, bringing such refinements as those just alluded to -- compositionality, etc. -- can be found in \cite[Chp. 2]{French}, where further references to the extensive literature on this topic are also supplied.} If working with logics as sets of formulas  -- as we shall be here --  rather than as consequence relations,  the previous definitions apply by deleting everything to the left of the $\vdash$. A number of authors have assumed that such translations also preserve other syntactic (or proof-theoretic)  properties. One well-known example is Crossley's mistake in \cite[p. 19]{Crossley}, which resulted in the incompleteness of a putative axiomatization of classical predicate logic presented there, and another appears with this same effect, as we shall see below (Section \ref{Incompleteness}), in the celebrated Bell and Slomson \cite{Bell&Slomson}.\footnote{This problem does not, of course, affect the correctness of the model-theoretic results that constitute the core of the book.}
 Halmos in \cite{Halmos} had proposed an axiomatization of propositional calculus via translation that turned out to be incomplete. The inadequacies were noticed by Hi{\. z} in  \cite{Hiz}, and later Frank \cite{Frank} attempted a generalization of Hi{\. z}'s observation, which itself ran into difficulties, several of them noted by  
Shapiro \cite{Shapiro}; further problems with Frank's discussion will be described below.\footnote{The oversight is briefly alluded to in \cite{CorcoranShapirofirst}; and more fully (see p.\,85) in \cite{CorcoranShapirosecond}. Corcoran and Shapiro between them produced three papers mentioning this point \cite{CorcoranShapirofirst, CorcoranShapirosecond, Shapiro},  but curiously none mentions either of the other two. A Spanish translation of \cite{CorcoranShapirosecond} incorporated some typographical corrections: \cite{CorcoranShapiro3}. Note also that we refer to the current mistake as Crossley's oversight without mentioning his coauthors, because the different chapters of \cite{Crossley} were written by different authors and the present issue arises in a Chapter by Crossley. (This may explain -- which is not to say \textit{justify} -- the lack of uniformity in format and style noted at p.\,93 of \cite{CorcoranShapirosecond}).} We shall be supplementing the discussion in \cite{Shapiro} with our own criticism of \cite{Frank} below in Section \ref{franksection}. Shapiro's key point was that na\"ively axiomatizing  $[\neg, \wedge, \exists]$ via translation in terms of the primitives  $[\neg, \rightarrow, \forall]$ is not possible in general because of incompleteness already in the \textit{propositional} fragment. He showed that certain propositional tautologies are not provable in Crossley's system (specifically, certain instances of $(\phi \wedge \phi) \rightarrow \phi $); Shapiro also notes a specifically quantificational deficit in the axiomatization: see note \ref{Shapironote}.

Crossley’s mistake can be quickly corrected, as Shapiro \cite[p. 249, note 3]{Shapiro}  remarked. In \cite{Crossley}, the primitives were  $[\neg, \wedge, \exists]$ with $[\rightarrow, \forall]$ being defined in the usual way, namely $\phi \rightarrow \psi \stackrel{\text{def}}{=}\neg (\phi \wedge \neg \psi)$ and $\forall x \phi(x) \stackrel{\text{def}}{=}\neg \exists x \neg \phi(x)$,  but the axiomatization was one for $[\neg, \rightarrow, \forall]$.   One might ask, though, whether predicate calculus with the primitives $[\neg, \wedge, \exists]$ can be axiomatized via translation in terms of $[\neg,  \rightarrow, \forall]$ when the propositional fragment is indeed complete (by adding, for example, the axiom schema $(\phi \wedge \psi) \leftrightarrow \neg(\phi \rightarrow \neg \psi)$, as suggested by Shapiro).\footnote{\label{Shapironote}In fact, it is not entirely transparent that this would do exactly what Shapiro probably had in mind (because of the subtleties involved in spelling out the meaning of $\leftrightarrow$ in this context), but at any rate the idea is clear: to make $(\phi \wedge \psi) $ replaceable with $ \neg(\phi \rightarrow \neg \psi)$. Hence, in the interest of simplicity, we will set this issue aside. Even though Shapiro suggests also adding the axiom schema $\neg \exists x \neg \phi \leftrightarrow \forall x \phi$, he does not show why this is \emph{necessary}, which is what we do in the present note.   Hi{\. z} \cite{Hiz} ends his second last paragraph with these words: `A translation of a complete set of axioms to another set of primitives would be complete only if from the resulting axioms the definitions of the first set of primitives followed.' Probably Shapiro is following this advice in the simplest way possible: to make sure it is provable, take it as an additional axiom. What we show here, though, is precisely why, in detail, things go wrong when the advice is not followed: namely that the replacement of equivalents -- `congruentiality' as it is called in the propositional (esp.\ modal) case, originally in \cite{Makinson}  -- breaks down in the defective axiomatizations.}   In this note, we explore the issue of the failure of the `replacement of equivalents' property that is at the root of both Hi{\. z}'s and Shapiro's examples, with special attention to the failure of provably equivalent open formulas to be interreplaceable within the scope of quantifiers. Hence, the difficulty with axiomatizing via translation is not a purely propositional issue, despite the literature containing mostly discussions focusing on what goes wrong only at the propositional level. As with the axiomatization in Crossley \cite{Crossley}, so also the presentation given in the celebrated text Bell and Slomson \cite{Bell&Slomson} is afflicted by this incompleteness problem. In \cite{Bell&Slomson} the primitives are also also  $[\neg, \wedge, \exists]$, while the axioms are given via translation in terms of $[\neg, \wedge, \vee, \rightarrow, \forall]$. By contrast with the case of \cite{Crossley}, in this case the axiomatization is complete for the propositional fragment.\footnote{In personal communications, both John Bell and Alan Slomson have confirmed that this problem was unknown to them.} However, the axiomatization in \cite{Bell&Slomson} would, of course, be complete if the choice of primitive quantifier had instead been $ \forall$.\footnote{Similarly, the axiomatization in \cite{Crossley} is complete if the choice of primitives had been $[\neg, \rightarrow, \forall]$. The effect of the generalization rule is obtained here by taking as instances of the axiom schemata all possible generalizations of the formulas directly instantiating those schemata, where this covers not only their complete universal closures but also the result of universally quantifying any number (including zero) of the free variables in those formulas.}

The following section contains our discussion of Frank's article  \cite{Frank} touched on above, and still having some current significance for the present topic. Then we turn, in Section \ref{Incompleteness}, to the topic of incomplete axiomatizations of predicate logic resulting from insufficient attention to the way completeness depends on choosing axioms and rules that are appropriate for the logical primitives. We are especially concerned with translating an axiomatization of classical predicate logic suited to one set of primitives into an axiomatization putatively complete for another set. Recent work~\cite{Casanovas,Kennedy} and the not so recent \cite{McGee} work on logical operations in predicate logic -- going back, in particular, to \cite{Henkin} -- provide a replacement for  the use of matrix methodology in showing unprovability (and hence independence and incompleteness) in propositional logic.

\section{Problems with Frank's note}\label{franksection}

On the first page of his note, Frank \cite{Frank} gives us the background to Hi{\. z}'s discussion, telling us that in his 1956 paper \cite{Hiz}, 
\begin{quotation}\small 
\noindent Halmos takes the Hilbert--Ackermann axioms for a sentential logic of $\mathord{\sim}$ and $\lor$, and the rule of inference 
\[  \RULE{p \,\lor q \qquad {\mathord{\sim} p}}{q}\]
and provides an axiom system for $\mathord{\sim}$ and $\&$ by means of the definition
\[p \lor q \leftrightarrow \mathord{\sim}(\mathord{\sim} p \mathop{\&} \mathord{\sim} q).\]\end{quotation}
Given notational conventions more widely prevailing today, it would now be better to write $p, q$ as $A, B$ or $\phi, \psi$ to make it clear that these are schematic letters for arbitrary formulas, rather than specifically propositional variables (or sentence letters), but it would have been  better to avoid $\leftrightarrow$ in the displayed definition and stuck to the formulations used by Halmos and Hi{\. z}, since the $\leftrightarrow$-formulation misleadingly suggests that an additional (biconditional) connective is somehow involved, raising questions about how this is related to the chosen primitives $\mathord{\sim}$ and $\lor$ for Hilbert and Ackermann, or $\mathord{\sim}$ and $\&$ for Halmos.\footnote{Our preferred way of recording a definition is with the `$\stackrel{\text{def}}{=}$' notation, so as to preserve as much neutrality as we can among competing accounts of what a definition (for logical vocabulary, in particular) is and does -- such as for example the metalinguistic and object-linguistic conceptions of definition contrasted in {\S}3 of \cite{Humberstone2} (or 3.16 in \cite{Humberstone3}). To minimize disruption, though, we prefer not to raise such issues every time we echo one of the authors discussed here in calling a $\leftrightarrow$-schema a definition. See also note~\ref{Shapironote} above.}  A more substantive issue arises over the rule Frank tells us Hilbert and Ackermann use. A glance at the actual text on p.~28 of \cite{Hilbert&Ackermann} reveals that the rule they employ takes us, not from premises $\phi \lor \psi$ and $\mathord{\sim}\phi$ to conclusion $\psi$, but rather from $\mathord{\sim}\phi \lor \psi$ and $\phi$ to $\psi$. But the substantive question this raises -- as to whether this change makes a difference to the set of provable formulas -- is not one we need to consider here, since it does not bear on Frank's commentary on Hi{\. z}

 Hi{\. z}'s paper uses a three-element matrix with tables for $\land$ (as we shall now write in place of $\&$) and $\neg$ (as we shall write for $\mathord{\sim}$) that validates all theorems forthcoming on the basis of Halmos's $\{\lor,\neg\}$-axiomatization but not all classical tautologies in the $\{\land,\neg\}$-fragment, showing the axiomatization to be incomplete, despite the completeness of the $\{\lor,\neg\}$-axiomatization  of Hilbert and Ackermann on which it was based, replacing every $\phi \lor \psi$ in the latter by $\neg(\neg \phi \land \neg \psi)$. The main point of Frank's discussion is that this ingenious matrix argument was not needed, because a simpler general observation already suffices to show that the $\{\land,\neg\}$-fragment of classical propositional logic could not possibly have been completely axiomatized by Halmos's axiomatization. And this general observation appears more or less as follows (we may call this `Frank's Claim'):\footnote{We shall replace Frank's  notation `${\rm A}1,\ldots, {\rm A}N$' and `${\rm R}1,\ldots,{\rm R}M$', for axioms and rules respectively, with `$A_1,\ldots,A_n$' and `$R_1,\ldots, R_m$'. The same passage is quoted, though in this case verbatim, on the opening page of \cite{Shapiro}.}

\begin{quotation}\noindent \small If {\sf T}($A$) is the closure of a formal system in a language $\mathcal{L}$, with axioms $A_1,\ldots,A_n$ and rules $R_1,\ldots, R_m$, and {\bf t} a rule of translation from $\mathcal{L}$ to $\mathcal{L}'$, then {\sf T}$'$, the closure of ${\bf t}(A_1),\ldots,{\bf t}(A_n),{\bf t}(R_1),\ldots, {\bf t}(R_m)$, is equal to {\bf t}(${\sf T}(A)$).\end{quotation}

\noindent  Frank remarks that a proof by induction (on the number of rule applications in a proof) of this claim is facilitated by noting that for a $k$-premise rule $R_j$, and formulas $\phi_1,\ldots,\phi_k, \psi$:\begin{center}
${\bf t}(R_j) \, = \, \{\langle {\bf t}(\phi_1),\ldots,{\bf t}(\phi_k),{\bf t}(\psi)\rangle \,\vert\,\langle \phi_1,\ldots \phi_n, \psi\rangle \in R_j\}$\end{center}
though this is best taken, not as a comment about the translation of rules -- an otherwise unexplained notion -- but as a definition of what it is to apply {\bf t} to rules (formulated for $\mathcal{L}$), in terms of the initially specified {\bf t} as applied to formulas (of $\mathcal{L}$), identifying a $k$-premise rule with the set of  all tuples $\langle \phi_1,\ldots \phi_n, \psi\rangle$ that constitute an application of the rule, the $\phi_i$s being premises and $\psi$ the conclusion. Shapiro \cite[p.\,249]{Shapiro},  points out that this definition does not in fact capture what most people have had in mind in translating rules, because it ignores the role of schematic letters in their formulation which, when interpreted over $\mathcal{L}'$, are taken as ranging over all formulas of $\mathcal{L}'$ and not just those of the form ${\bf t}(\phi)$ for some formula $\phi$ of $\mathcal{L}$\footnote{This point of Shapiro's suggests that instead of identifying a rule with the set of its applications in a particular language, we think of it as mapping any language equipped with the logical vocabulary governed by the rule to the relevant application-set in that language. This is the policy urged in 4.33 of \cite{Humberstone3}.}. This consideration also applies to 0-premise rules (axiom schemata). This makes such a rule a $(k + 1)$-ary relation between formulas -- and not necessarily a functional such relation.\footnote{We include this remark because at the top of the second page of his note, Frank adds that ${\bf t}(R_j)(y_1,\ldots,y_k) = {\bf t}(x) = {\bf t}(R_j)({\bf t}(y_1),\ldots,{\bf t}(y_k))$, in which he writes $y_1,\ldots,y_k,x$ for what appear as $\phi_1,\ldots,\phi_k,\psi$ above -- and incidentally writes `${\rm R}J\langle y_1,\ldots,y_k \rangle$' (etc.) for what  was just  quoted as `$R_j(y_1,\ldots,y_k)$'. There is no reason to restrict attention to such `functional' rules, however. For example, one could have a rule in a Hilbert-style/axiomatic setting like the natural deduction rule of $\lor$-introduction of a second disjunct taking us for any formulas $\phi, \psi$ from the premise $\phi$ to the conclusion $\phi \lor \psi$, which is the binary relation comprising all pairs $\langle \phi, \phi \lor \psi\rangle$ for $\phi, \psi$ in the language concerned, so the the conclusion is not uniquely determined by the premise.} There is good news and bad news. The good news is that this issue about the functionality of the rules does not affect the inductive argument Frank has in mind. The bad news is that the induction fails for a subtle reason explained in Shapiro \cite[p.\,348]{Shapiro} with the aid of a simple counterexample: we need to impose the condition that the translation itself should be injective. Frank's Claim -- setting to one side the need for a corrected formulation -- has yet to be brought to bear on the case of Halmos's incomplete axiomatization, however, so let us see how this is attempted in Frank's discussion.

We shall quote the passage in question verbatim, except for symbolizing negation and conjunction by $\neg$ and $\land$; this includes reproducing the phrase `the domain {\sf D}', even though this `{\sf D}' appears nowhere else in the paper (and the syntax is obscure: the domain {\sf D} is not \textit{what}?):
\begin{quote}\small Thus, if {\bf t} is not an onto-mapping from $\mathcal{L}$ to $\mathcal{L}'$, (as the domain {\sf D} is not, having as its range in the language containing $\neg$ and $\land$ only sentences beginning with $\neg$), a complete axiomatization in $\mathcal{L}$ will result in an incomplete 
one in $\mathcal{L}'$.\end{quote}
Thus it seems that Frank has switched to using `{\bf t}' not as a variable for discussing translations in general, but to allude to the specific {\bf t} in play in Halmos's discussion. In mentioning the case as one in which {\bf t} does not map $\mathcal{L}$ onto (but only into) $\mathcal{L}'$, Frank has usefully made explicit the fact that {\bf t} is a mapping -- as well as raising the question of how its lack of surjectivity might bear on the current issue, something we shall put on hold for a moment -- and has also explicitly identified its domain and codomain.\footnote{\label{t*} Ideally, this second use of `{\bf t}' for a mapping from $\mathcal{L}$-rules to $\mathcal{L}'$-rules would be notationally distinguished -- as {\bf t}* --   say, from the formula-to-formula map {\bf t} it is induced by, though here we have been following Frank in suppressing the distinction (omitting the  `*', on the suggestion just mooted). In Section \ref{Intro} we noted that since we were working with logics sets of formulas we would not be using the apparatus of consequence relations and could ignore everything to the left of the `$\vdash$' in our opening paragraph. It might seem that in discussing Frank's translations of axiomatizations the transition from $A$ to ${\sf T}(A)$ -- the closure of the set of axioms under the rules -- which he wants his translations to preserve, precisely reinstates consequence relation (or more accurately, the corresponding consequence operation) in play in our opening paragraph in Section \ref{Intro}. This is not straightforwardly so, however, because the notation $A$ (which appears in the passage quoted from Frank, though without being properly introduced there) stands for a set containing not only axioms but also rules. (Compare the `tuple systems' of \cite[subsec. 0.26]{Humberstone3}.) If we instead think of $A$ as the set of axioms and build the use of the rules into the ``{\sf T}'' part of the ``${\sf T}(A)$'' notation, then we will have a genuine consequence operation, though typically it will not correspond to the consequence relation most readily associated with the logic in question, because of the \textit{rules of proof} vs \textit{rules of inference} contrast. (See the index entry under that heading in \cite{Humberstone5} for discussion and references.)} 

In view of this, one might ask why Frank writes (as quoted above) that the range of {\bf t} comprises only formulas (or, as he says, sentences) beginning with $\neg$. What, in particular, is ${\bf t}(p_i)$ supposed to be, for a sentence letter/propositional variable $p_i$, in the case of the translations {\bf t} currently under consideration? In the literature on translations embedding one logic in another (whether the logics are taken as sets of theorems or as consequence relations or \ldots) as opposed to the associated translations -- {\bf t}* in note \ref{t*} --  from one proof system to another (whether, as for the current discussion, a Hilbert-style  system, or a natural deduction system or a Gentzen system), what get called \textit{definitional} translations have to satisfy two conditions: they have to be \textit{variable-fixed}, i.e., satisfy ${\bf t}(p_i) = p_i$ for all propositional variables $p_i$ (where we here restrict attention  restrict attention to sentential languages for simplicity, and assume that all are equipped with the same countable supply of such variables), as well as being \textit{compositional} (sometimes called `schematic') in the sense that for every primitive $n$-ary connective $\#$ of $\mathcal{L}$ there is a formula $\phi(p_1,\ldots,p_n) \in \mathcal{L}'$ containing only the sentence letters displayed, for which we have:\begin{center}
for all $\psi_1,\ldots, \psi_n \in \mathcal{L}$, ${\bf t}\big(\#(\psi_1,\ldots, \psi_n)\big) = \phi\big({\bf t}(\psi_1),\ldots,{\bf t}(\psi_n)\big)$.\end{center}
One can think of $\phi(p_1,\ldots,p_n)$ as putatively defining the $\#$ of $\mathcal{L}$ in the language $\mathcal{L}'$, which is why these are called {\it definitional} translations; in the case of Halmos $\mathcal{L}$ has connectives $\neg$ and $\lor$ and $\mathcal{L}'$ has connectives $\neg$ and $\wedge$, and the inductive definition of the {\bf t} involves:
\begin{itemize}
\setlength\itemsep{.3em}
\item ${\bf t}(p_i) = p_i$
\item ${\bf t}(\neg \phi) = \neg({\bf t}(\phi))$
\item ${\bf t}(\phi \lor \psi) = \neg(\neg {\bf t}(\phi) \land \neg{\bf t}(\psi))$
\end{itemize}
In view of these considerations, we are inclined to think of Frank's comment about the range of {\bf t} comprising only formulas of the form $\neg \phi$ as an oversight -- but in fact a revealing one if the charge in the following paragraph of a confusion between $\mathcal{L}$ and $\mathcal{L}'$ as, on the one hand languages, and on the other hand, logics formulated in those languages, is correct: since certainly in the setting of Halmos's and Hi{\. z}'s discussion, the $p_i$ are not going to show up as \textit{provable} formulas in the range of {\bf t}. Note, incidentally, that the issue raised by Shapiro  concerning injectivity is not addressed by insisting on definitional translations, since we could easily have such a translation that is not injective. One way would be to choose the same `defining' formula $\phi$ for two connectives of the same arity. But we return from injectivity to the matter of surjectivity.\footnote{\label{constants_note}The example, not summarized above, of a non-injective translation on the first page of Shapiro's paper can be presented as a definitional translation subject to the convention that the propositional variables come in countable supply by treating his $a, b, c$ in $\mathcal{L}$ and $A, B$ in $\mathcal{L}_2$ as nullary connectives (sentential constants); the simultaneous presence in the languages of the $p_i$ ($i \in \omega$) does not affect the example.}

What is not obvious is why a failure of surjectivity on {\bf t}'s part should occasion a failure of completeness for the target logic -- axiomatized by applying {\bf t} to a complete axiomatization of the source logic. Recall that  $\mathcal{L}$ and $\mathcal{L}'$ are not the source and target logics involved here, but rather just the languages of these logics. So one cannot immediately reason: 
\begin{quote}\small Suppose that $\phi' \in \mathcal{L}'$ is not ${\bf t}(\phi)$ for any $\phi \in \mathcal{L}$. In that case the target logic must be incomplete, since ${\bf t}(\phi)$ is cannot be provable in it, by the main observation above (beginning ``If {\sf T}($A$) is the closure\ldots'').\end{quote} 
This would not work, because the notion of completeness in play here is most evidently explicated in semantic terms: we are trying to axiomatize the classically valid (`tautologous') formulas in the language $\mathcal{L}$, so what we need is not just some formula or other of $\mathcal{L}'$ that is not ${\bf t}(\phi)$ for any formula $\phi \in \mathcal{L}$ -- all that a failure of surjectivity asserts -- but that we have some \textit{valid} formula $\phi'$ of $\mathcal{L}'$ that is not ${\bf t}(\phi)$ for any (here redundantly: valid) formula $\phi \in \mathcal{L}$. With this apparent strengthening of (a correctly formulated version of) Frank's Claim, we could proceed to the desired incompleteness conclusion, since such a $\phi'$ would then be a witness to the incompleteness of the axiomatization obtained by applying {\bf t} to the initially given complete axiomatization of the valid formulas of $\mathcal{L}$.

To see how this gap arises between Frank's claim to have provided a simpler alternative to Hi{\. z}'s conclusion and the explicit justification he provides for that claim, it is necessary to inquire more deeply into what unstated conditions might be in play concerning the notion of a translation beyond it being a mapping from one language to another (continuing here, for simplicity, to identify a language with its set of formulas). In many cases, once attention is restricted to definitional translations, the absence of a $\phi'$ which is not (classically) valid can be exploited to find a related formula of $\mathcal{L}'$ which is valid. For example, take $\mathcal{L}$ and 
$\mathcal{L}'$ to be the languages of classical propositional logic with primitives $\{\lor, \neg\}$ in the former case and $\{\lor, \neg, \bot\}$ in the latter, with {\bf t} the identity translation: a degenerate case of a definitional translation:
\begin{itemize}
\setlength\itemsep{.3em}
\item ${\bf t}(p_i) = p_i$
\item ${\bf t}(\neg \phi) = \neg({\bf t}(\phi))$
\item ${\bf t}(\phi \lor \psi) = {\bf t}(\phi) \lor {\bf t}(\psi))$
\end{itemize}
Evidently $\bot$ is not in the range of {\bf t}, and while this is not an immediate threat to the completeness of the result of applying {\bf t} to any complete axiomatization of the valid formulas of $\mathcal{L}'$ since $\bot$ is not such a formula, we note that this means that we also have, for example, $\neg p \lor (p \lor \bot)$ as a valid $\mathcal{L}'$-formula that is missing from the range of {\bf t}, occasioning incompleteness. What if the $\phi'$ not in the range of {\bf t} has no such distinctive logical behaviour -- for example is a nullary connective (cf.\ note \ref{constants_note}) like $\bot$ as it behaves in Johansson's Minimal Logic, while $\neg$ and $\lor$ continue to behave classically? Then, in place of the $\phi'$ in question we could use the disjunction $\phi' \lor \neg \phi'$ as the missing valid formula. Since we are focusing on Hilbert-style systems here -- and not including as logics `purely inferential' or atheorematic consequence relations (such as the conjunction--disjunction fragment of classical logic),  the logics concerned need to have at least one provable formula. So we simply make judicious substitutions for any propositional variable occurring in such a formula. If we were concerned with intuitionistic propositional logic, for instance, we could use $\phi' \to \phi'$ or $\neg\neg(\phi' \lor \neg \phi')$ to play this role. 

This does not cover all eventualities, however. Saying that we require at least one provable formula does not guarantee that there is such a formula in which there occur propositional variables for which the envisaged substitution of $\phi'$ for some $p_i$ can be made. So again making use of constants (and once more the example mentioned in note \ref{constants_note} is relevant), let $\mathcal{L}$ and $\mathcal{L}'$ have for their logical vocabulary $\{\top\}$ and $\{\top, \bot\}$ and let {\bf t} be the identity map. The (classically) valid formulas of $\mathcal{L}$ are one in number: {$\top$}, so we can axiomatize the logic with that formula as our sole axiom, and no rules at all. The translation under {\bf t} of this axiomatization is the identity and so suffers from no incompleteness, even though Frank's sufficient condition for producing incomplete translations is satisfied: {\bf t} is not a surjective translation (though we cannot construct a valid formula to exploit this). Of course, this example differs from those considered earlier among fragments of classical logic in that we are not dealing with a functionally complete fragment (to say the least). But the role of functional completeness is not particularly emphasized in Frank's discussion and its role has not been investigated here either. 

\section{Incompleteness}\label{Incompleteness}
\subsection{The case of one-variable classical predicate calculus}\label{one}
Let us look now at the simple example of the monadic one-variable classical predicate calculus studied by Henkin in \cite[p.  6]{Henkin}; the reason for beginning with the one-variable fragment is explained below (in the paragraph following Remark \ref{alt}). We shall show that  this logic with the primitives $[\neg, \rightarrow, \exists]$ cannot, in general, be axiomatized via translation in terms of $[\neg,  \rightarrow, \forall]$ (our choice of $\neg, \rightarrow$ is just for simplicity, the same result holds for $\neg, \rightarrow, \wedge, \vee$). As Henkin presented his one-variable system, one allows only unary vocabularies and one individual variable $x$. This logic is axiomatizable by the usual presentations of  predicate calculus (such as that in \cite{Mendelson}).

Consider then the following standard\footnote{See \cite{Bell&Slomson} where the authors also add axiom schemata for the remaining connectives. The results in what follows all remain true if we add those further  postulates as well.} attempt at axiomatization via translation where we set $\forall x \phi(x) \stackrel{\text{def}}{=}\neg \exists x \neg \phi(x)$:

\begin{itemize}
\item[] {\sc Axiom schemata}
\item[(A1)] $\phi \rightarrow (\psi \rightarrow \phi)$
\item[(A2)] $(\phi \rightarrow (\psi \rightarrow \chi)) \rightarrow ((\phi \rightarrow \psi) \rightarrow (\phi \rightarrow \chi))$
\item[(A3)] $(\neg \phi \rightarrow \neg \psi) \rightarrow (\psi \rightarrow \phi)$
\item[(A4)] $\forall x \phi(x) \rightarrow \phi(x)$
\item[(A5)] $\forall x (\phi \rightarrow \psi (x)) \rightarrow (\phi \rightarrow \forall x \psi(x))$ where $x$ is not free in $\phi$.

\end{itemize}

\begin{itemize}
\item[] {\sc Rules}
\item[(R1)]  \emph{Modus Ponens}:  
\[  \RULE{\phi \rightarrow \psi  \qquad \phi}{\psi}\]
\item[(R2)] \emph{Generalization}: \[  \RULE{\phi}{\forall x \phi}\]
\end{itemize}

When these rules and axiom schemata are given for predicate calculus with infinitely many variables, $x$ simply serves as a place holder for any of the variables in the official list $x_1, x_2, x_3, \dots$
Our strategy consists in adapting the argument for modal logic from \cite[Corollary 2.2]{Humberstone}  to the present setting. Roughly, we shall be  interpreting the logical primitives of one-variable predicate calculus by    \emph{operations} on domains (in the sense of \cite{McGee} and more recently \cite[Def. 6.1]{Casanovas} or \cite[\S 2.1]{Kennedy}) that diverge from the standard in order to show the unprovability of the validity $\exists x \phi (x) \rightarrow \exists x \neg \neg\phi (x)$. This method was in fact, introduced already by Henkin in  \cite[p. 21]{Henkin} using the term `generalized models'\footnote{Not to be confused with Henkin’s ‘general models’ for second order logic, though the two ideas have something in common, each giving an unintended interpretation to some of the logical vocabulary.} to show the incompleteness of certain finite-variable logics with respect to obvious attempts at axiomatization. This kind of argument is a generalization to the first-order level of the typical proofs of independence of different axioms for propositional calculus.

\begin{Rmk}\label{logop}\emph{ To illustrate the notion of a logical operation on a domain $A$ we shall refer the reader to \cite[Example 3]{Kennedy}.  For example, in the context of monadic logic, if we have  $X_0, X_1 \subseteq A$, then conjunction is the operation $f_\wedge (X_0, X_1) \mapsto X_0 \cap X_1$, negation is the operation $f_\neg (X_0) \mapsto A \setminus X_0$ and the existential quantifier $\exists $ is the operation 
\begin{center}$
    f_\exists (X_0) =
    \begin{cases*}
    A& if $X_0 \neq \emptyset$ \\
  \emptyset      & otherwise
    \end{cases*}$
    \end{center}
}\hspace*\fill$\blacktriangleleft$
\end{Rmk}

\
\

Using the  notion of an operation on models as exemplified in Remark \ref{logop}, one can easily obtain a satisfaction relation. Then we may take the \emph{value} of a formula $\phi$  on a model $\mathfrak{A}$, denoted $\mathfrak{A}(\phi)$, to be the set $\{ a \in A \mid \mathfrak{A} \models \phi[a]\}$ and it can be computed recursively using the  operations corresponding to the primitives of the logic. For a sentence, $\mathfrak{A}(\phi)$ is always $A$ or $\emptyset$.
In this sense, a first-order sentence might be said to be \emph{logically valid} if its value is the whole domain in every model.  Once we reinterpret the logical primitives in new ways, what $\mathfrak{A}(\phi)$ ends up being will change.

We show the independence of $\exists x\phi(x) \to \exists x \neg\neg\phi(x)$ in the course of the proof of Proposition \ref{monadic}, though readers for whom the `logical operations' approach summarised in Remark 2 may prefer to glance first at Remark 3 below to see the proof would go without explicitly invoking that apparatus.

\begin{Pro}\label{monadic} Let $\vdash$ stand for provability in the axiomatization of one-variable logic given above. Consider a vocabulary  $\tau= \{P\}$  where $P$ is a unary predicate letter. There is a two-element model $\mathfrak{A}$, with domain $A=\{u, v\}$, and an interpretation of the primitives $[\neg, \rightarrow, \exists]$ such that $\vdash \phi$ only if  either $\mathfrak{A}(\phi) = A$ or $\mathfrak{A}(\phi) = \{u\}$. Moreover, we have that  $$\mathfrak{A}(\exists x P (x) \rightarrow \exists x \neg \neg P (x)) = \emptyset.$$

\end{Pro}

\begin{proof} Given the domain $A= \{u, v\}$, we proceed to re-interpret the operations corresponding to $\rightarrow, \neg$ and $\exists$ in this domain:

\
\FloatBarrier

\begin{table}[h]
\begin{tabular}{l||cccc}
$f^*_\rightarrow (X_0, X_1)$ &  $X_1 = A$   &   $X_1 = \{u\}$   &  $X_1 = \{v\}$          &   $X_1 = \emptyset$     \\ \hline\hline
                     $X_0 = A$         & $A$ & $\{u\}$ & $\{v\}$ &  $\emptyset$ \\
                  $X_0 = \{u\}$            &  $A$ & $A$ & $\{v\}$  & $\{v\}$  \\
                $X_0 = \{v\}$             &$A$  & $\{u\}$ &  $A$ & $\{u\}$ \\
                    $X_0 = \emptyset$         &  $A$ &  $A$ &  $A$  & $A$
\end{tabular}
\end{table}

    \FloatBarrier
    
\begin{minipage}{.45\linewidth}
\[
    f^*_\exists (X_0) =
    \begin{cases*}
    A& if $X_0 = A$ \\
A   &  if $X_0 = \{u\}$  \\
    A   &  if $X_0 = \{v\}$\\
      \emptyset      &  if $X_0 = \emptyset$    \end{cases*}\]
\end{minipage}  
\begin{minipage}{.45\linewidth}
\[
    f^*_\neg (X_0) =
    \begin{cases*}
  \emptyset   & if $X_0 = A$ \\
\emptyset     &  if $X_0 = \{u\}$  \\
  \{u\}  &  if  $X_0 = \{v\}$\\
      \{u\}    &  if $X_0 = \emptyset$    \end{cases*}    \]
\end{minipage}   

\

Under our earlier definition of $\forall$, then the operation for this defined symbol is:

  \begin{center}
\[
  f^*_\neg(  f^*_\exists (f^*_\neg (X_0))) = f^*_\forall (X_0) =
    \begin{cases*}
    \{u\}& if $X_0 = A$ \\
\{u\}   &  if $X_0 = \{u\}$  \\
    \emptyset  &  if $X_0 = \{v\}$\\
      \emptyset      &  if $X_0 = \emptyset$    \end{cases*}\]
\end{center}   
    
    \
    
     Take the model $\mathfrak{A}$ that results from the domain $\{u, v\}$ and letting   $\mathfrak{A}(P (x))=\{v\}$.
     
     First, one can see that 
    the logically valid (in the standard sense) sentence $\exists x P (x) \rightarrow \exists x \neg \neg P (x)$ takes value $\emptyset$ in $\mathfrak{A}$, which implies that it is not derivable from the axiomatization we just considered. To see this, it suffices to  check the following equalities: 

    $$f^*_\rightarrow ( f^*_\exists (\{v\}), f^*_\exists ( f^*_\neg (f^*_\neg(\{v\})))) = f^*_\rightarrow ( A, f^*_\exists ( f^*_\neg (  \{u\} ))) = f^*_\rightarrow ( A, f^*_\exists (      \emptyset  )) =  f^*_\rightarrow ( A,      \emptyset  ) =  \emptyset.$$

      On the other hand,  every axiom of the axiomatization via translation in terms of $[\neg,  \rightarrow, \forall]$ takes as value either $\{u, v\}$ or $\{u\}$. We check this for A4 and A5 and leave the rest as exercises for the reader.  For A4, first observe 
    
    $$ f^*_\forall(\mathfrak{A}(\phi)) =  \begin{cases*}
    \{u\}& if $\mathfrak{A}(\phi) = A$ \\
\{u\}   &  if $\mathfrak{A}(\phi) = \{u\}$  \\
    \emptyset  &  if $\mathfrak{A}(\phi) = \{v\}$\\
      \emptyset      &  if $\mathfrak{A}(\phi) = \emptyset$    \end{cases*}$$
      which means that we have the following  table for $f^*_\rightarrow (  f^*_\forall(\mathfrak{A}(\phi)), \mathfrak{A}(\phi))$ (and, hence the operation always takes $A$ as  value):
      
      \FloatBarrier

\begin{table}[h]
\begin{tabular}{l||cccc}
$f^*_\rightarrow (  f^*_\forall(\mathfrak{A}(\phi)), \mathfrak{A}(\phi))$ &  $\mathfrak{A}(\phi) = A$   &   $\mathfrak{A}(\phi) = \{u\}$   &  $\mathfrak{A}(\phi)= \{v\}$          &   $\mathfrak{A}(\phi) = \emptyset$     \\ \hline \hline
                  $f^*_\forall(\mathfrak{A}(\phi)) = \{u\}$            &  $A$ & $A$ &  &   \\
                    $f^*_\forall(\mathfrak{A}(\phi)) = \emptyset$         &  &  &  $A$  & $A$
\end{tabular}
\end{table}
      
      \FloatBarrier
For A5, we must compute the values of $$f^*_\rightarrow (  f^*_\forall(f^*_\rightarrow(\mathfrak{A}(\phi),  \mathfrak{A}(\psi(x)))), f^*_\rightarrow (\mathfrak{A}(\phi),  f^*_\forall (\mathfrak{A}(\psi(x))) )).$$  First observe that given that $\phi$ is a sentence (so this sentence only takes one of the values $A$ or $\emptyset$), the table for $f^*_\rightarrow(\mathfrak{A}(\phi),  \mathfrak{A}(\psi(x)))$ gets simplified:


\begin{table}[h]
\begin{tabular}{l||cccc}
$f^*_\rightarrow(\mathfrak{A}(\phi),  \mathfrak{A}(\psi(x)))$ &  $\mathfrak{A}(\psi(x)) = A$   &   $\mathfrak{A}(\psi(x)) = \{u\}$   &  $\mathfrak{A}(\psi(x)) = \{v\}$          &   $\mathfrak{A}(\psi(x)) = \emptyset$     \\ \hline\hline
                     $\mathfrak{A}(\phi) = A$         & $A$ & $\{u\}$ & $\{v\}$ &  $\emptyset$ \\
                    $\mathfrak{A}(\phi) = \emptyset$         &  $A$ &  $A$ &  $A$  & $A$
\end{tabular}
\end{table}
Then   
  $$ f^*_\forall(f^*_\rightarrow(\mathfrak{A}(\phi),  \mathfrak{A}(\psi(x)))) =  \begin{cases*}
    \{u\}& if $f^*_\rightarrow(\mathfrak{A}(\phi),  \mathfrak{A}(\psi(x)))= A$ \\
\{u\}   &  if $f^*_\rightarrow(\mathfrak{A}(\phi),  \mathfrak{A}(\psi(x))) = \{u\}$  \\
    \emptyset  &  if $f^*_\rightarrow(\mathfrak{A}(\phi),  \mathfrak{A}(\psi(x))) = \{v\}$\\
      \emptyset      &  if $f^*_\rightarrow(\mathfrak{A}(\phi),  \mathfrak{A}(\psi(x))) = \emptyset$    \end{cases*}$$
      
      Similarly, we have
      
        $$ f^*_\forall(\mathfrak{A}(\psi(x))) =  \begin{cases*}
    \{u\}& if $\mathfrak{A}(\psi(x)) = A$ \\
\{u\}   &  if $\mathfrak{A}(\psi(x)) = \{u\}$  \\
    \emptyset  &  if $\mathfrak{A}(\psi(x)) = \{v\}$\\
      \emptyset      &  if $\mathfrak{A}(\psi(x)) = \emptyset$    \end{cases*}$$ 
      
      and then 
      
      \FloatBarrier

\begin{table}[h]
\begin{tabular}{l||cccc}
$f^*_\rightarrow (\mathfrak{A}(\phi),  f^*_\forall (\mathfrak{A}(\psi(x))) )$ &     &   $f^*_\forall (\mathfrak{A}(\psi(x)))  = \{u\}$   &     &   $f^*_\forall (\mathfrak{A}(\psi(x)))  = \emptyset$     \\ \hline \hline
                     $\mathfrak{A}(\phi) = A$         &  & $\{u\}$ &  &  $\emptyset$ \\
                    $\mathfrak{A}(\phi) = \emptyset$         &   &  $A$ &   & $A$
\end{tabular}
\end{table}

\FloatBarrier

\

If we let $S_0$ stand for $  f^*_\forall(f^*_\rightarrow(\mathfrak{A}(\phi),  \mathfrak{A}(\psi(x))))$ and $S_1$ for $f^*_\rightarrow (\mathfrak{A}(\phi),  f^*_\forall (\mathfrak{A}(\psi(x))) )$, then we may build the following table:
\FloatBarrier
\begin{table}[h]
\begin{tabular}{l||cccc}
$f^*_\rightarrow ( S_0, S_1)$ &  $ \mathfrak{A}(\psi(x)) = A$   &   $ \mathfrak{A}(\psi(x)) = \{u\}$   &  $ \mathfrak{A}(\psi(x)) = \{v\}$          &   $ \mathfrak{A}(\psi(x)) = \emptyset$     \\ \hline \hline
                     $\mathfrak{A}(\phi) = A$         & $A$ & $\{u\}$  & $A$ &  $A$ \\
                    $\mathfrak{A}(\phi) = \emptyset$         &  $A$ & $A$  & $A$   & $A$
\end{tabular}
\end{table}

\FloatBarrier
Moreover, the rules of inference preserve these values: they take you from premises with values   $A$ or $\{u\}$ to conclusions with those very same values.
    \end{proof}

\begin{Rmk}\label{alt}\emph{ In the proof of Proposition \ref{monadic} we could have presented the new way of interpreting the primitives semantically by means of the following satisfaction relation $ \models^*$ where $\mathfrak{A}$ is taken to be a fixed parameter (namely the structure built in the proof):
\begin{itemize}
\item[] for all $a \in A$,
\item[] $\mathfrak{A} \models^* P(x)[a]$ iff $a \in \mathfrak{A}(P(x))$
\item[] $\mathfrak{A} \models^* \neg \phi [a]$ iff $a=u$ and $\mathfrak{A} \not \models^*  \phi [u]$ (so negation means failure to be satisfied by the distinguished element $u\in A$),
\item[] $\mathfrak{A} \models^*  (\phi  \rightarrow \psi) [a]$ iff $\mathfrak{A}\not \models^*  \phi  [a]$ or $\mathfrak{A} \models^*  \psi  [a]$,
\item[] $\mathfrak{A} \models^* \exists x \phi [a]$ iff $\mathfrak{A}  \models^*  \phi [b]$ for some $b \in A$,
\item[] $\mathfrak{A} \models^* \forall x \phi [a]$ iff  $a=u$ and $\mathfrak{A}  \models^*  \phi [u]$.
\end{itemize}
Then we may say a formula is $true^*$ in $\mathfrak{A}$ if satisfied by the distinguished element $u$. What the argument in Claim \ref{monadic} showed then is that all the axioms of
the calculus are $true^*$ in $\mathfrak{A}$, the  rules preserve $truth^*$ but $\exists x P (x) \rightarrow \exists x \neg \neg P (x)$ is not $true^*$. }
\hspace*\fill$\blacktriangleleft$
\end{Rmk}    
    
      By way of background, let us note that the above proof of Proposition \ref{monadic} is an adaptation of the proof of \cite[Corollary 2.2]{Humberstone}, in the setting of modal logic,\footnote{In the setting of modal logic, work on sensitivity to the choice of primitives began with \cite{Makinson}, though this concerned the structure of the lattice of all modal logics rather than the failure of `axiomatization by translation'. The latter theme is pursued in the case of intuitionistic logic in \cite{Humberstone2}; a combination of vocabulary unfamiliar to a copy editor and inadequate proof reading by the author resulted in the frequent appearance of `intuitionistic(ally)' in this paper as `intuitional(ly)'. Furher illustrations, from modal logic, of losing congruentiality as a result of changing primitives and not making compensatory adjustment can be found under Example 1.3.18 in \cite{Humberstone4}.} but with a semantics borrowing one aspect of the use of non-normal worlds. Usually those are exploited in two ways for non-normal modal logics, one being that the validity of a formula requires only its truth at the normal worlds in all models, with the other being that the normality of a world is an additional necessary condition for the truth of a $\Box$-formula at that world, over and above the truth of of the formula in the scope of the $\Box$ at all accessible worlds. The first of these roles is still played by the normal worlds of \cite{Humberstone}, but the second role is altered so that it is not $\Box$ (or $\Diamond$) formulas that make an additional demand of normality of the points at which they are evaluated, but $\neg$-formulas that require the world in question to be normal (as well as that the formula in the immediate scope of $\neg$ should fail to be true at the point). In addition, the accessibility relation in play is universal, and so does not need to be mentioned. This feature makes monadic one-variable predicate logic the appropriate non-modal analogue for present purposes, a closed monadic formula corresponding to a `fully modalized' modal formula, and the two-world models of \cite{Humberstone} -- with one normal and one non-normal world -- can then interpret any free occurrence of (the sole candidate) variable as picking out the normal element. This normal element  above is referred to as $u$, the non-normal element being $v$, exactly as in the notation of \cite{Humberstone} for the corresponding two worlds. The domain $A$ is there just the universe $\{u, v\}$ of the models under consideration. This set is labelled 1, with $\{u\}, \{v\}$, and $\emptyset$ appearing as 2, 3, and 4 in Figure 1 of \cite[p.399]{Humberstone} being the modal version of of the tables in the proof of Claim \ref{monadic} above, with $\Box$ and $\Diamond$ in place of $\forall$ and $\exists$. Note that all that is required of the nonstandard semantics in such cases is that the axiomatization under discussion should be \textit{sound} w.r.t.\ the notion of validity provided by the semantics, and that the formula whose unprovability is to be shown is invalid. It is not required that the axiomatization be not only sound but complete w.r.t.\ the semantics on offer, though for one reason or another, one may be interested in this. In the modal case, Omori and Skurt \cite{Omori&Skurt} explore the possibility of a semantic characterization of the `failed axiomatization' of the modal logic {\sf K} discussed in \cite{Humberstone}. And, by way of a non-modal example, Shapiro \cite[p.  249, last para.]{Shapiro} provides a semantic description of Crossley's (attempted) axiomatization in \cite[p. 19]{Crossley} of classical predicate logic.
    
    \subsection{The case of classical predicate calculus}
   Next we are going to use the approach in Remark \ref{alt}  to transport the incompleteness  result  to \emph{monadic} logic with infinitely many variables. Everything we do can be adapted to the polyadic case but the monadic case is simply easier to understand.\footnote{For example, one could set $\mathfrak{A}(R(x_1,...,x_m)) = \{\langle a_1,...,a_n\rangle \vert a_1 = ... = a_m = v\}$ and adjust the atomic clause for $\models^*$ appropriately.} This is just what suffices to refute  Bell and Slomson's claim that the axiomatization in \cite{Bell&Slomson} is complete.

  We retain the semantic apparatus of Section \ref{one}, now setting $\mathfrak{A}(P (x_i))=\{v\}$ ($i<\omega$). We start by defining the satisfaction relation $ \models^*$,\footnote{Following, for simplicity, the treatment of  variables and finite sequences of elements in a structure that appears in \cite{Marker}.} this time as follows (where $u$ is  the distinguished element of $ A$): 
    \begin{itemize}
\item[] for any formula $\phi(x_{i_1}, \dots, x_{i_n})$ with free variables among $\{x_{i_1}, \dots, x_{i_n}\}$,  a sequence $a_{i_1}, \dots, a_{i_n}$ of elements from $ A$ (in general with repetitions since $A$ has only two elements),
\item[] if $\phi(x_{i_1}, \dots, x_{i_n}) = P(x_{i_k})$, 
\item[]  $\mathfrak{A} \models^*  \phi[a_{i_k}]$ iff $a_{i_k} \in \mathfrak{A}(P(x_{i_k}))$,
\item[] if $\phi(x_{i_1}, \dots, x_{i_n}) = \neg \psi$, 
\item[]  $\mathfrak{A} \models^* \neg \psi [a_{i_1}, \dots, a_{i_n}]$ iff $a_{i_l}=u \ (1\leq l\leq n)$ and $\mathfrak{A} \not \models^*  \psi [\underbrace{u, \dots, u}_n]$,
\item[] if $\phi(x_{i_1}, \dots, x_{i_n}) =  \psi \rightarrow \chi$, 
\item[]  $\mathfrak{A} \models^*  (\psi  \rightarrow \chi) [a_{i_1}, \dots, a_{i_n}]$ iff $\mathfrak{A} \not \models^*  \psi  [a_{i_1}, \dots, a_{i_n}]$ or $\mathfrak{A} \models^*  \chi  [a_{i_1}, \dots, a_{i_n}]$,
\item[] if $\phi(x_{i_1}, \dots, x_{i_n}) =  \exists x_{j} \psi(x_{i_1}, \dots, x_{i_n},x_{j})$, 
\item[] $\mathfrak{A} \models^* \exists x_{j} \psi [a_{i_1}, \dots, a_{i_n}]$ iff $\mathfrak{A}  \models^*  \psi [a_{i_1}, \dots, a_{i_n}, b]$ for some $b \in A$,
\item[] if $\phi(x_{i_1}, \dots, x_{i_n}) =  \forall x_{j} \psi(x_{i_1}, \dots, x_{i_n}, x_{j}) = \neg \exists x_{j} \neg \psi(x_{i_1}, \dots, x_{i_n}, x_{j})$, 
\item[] $\mathfrak{A} \models^* \forall x_{j} \psi [a_{i_1}, \dots, a_{i_n}]$ iff $\mathfrak{A}  \models^*  \psi [\underbrace{u, \dots, u}_{n+1}]$ and $a_{i_l}=u \ (1\leq l\leq n)$.
\end{itemize}
    Then we may say a formula $\phi(x_1, \dots, x_n)$  is $true^*$ in $\mathfrak{A}$ if satisfied by the sequence $\underbrace{u, \dots, u}_n$ where $u$ is the distinguished element of $A$.

    \begin{Pro}\label{monadic2} Let $\vdash$ stand for provability in the axiomatization given in \emph{\cite{Bell&Slomson}}\footnote{Their axiomatization was roughly copied from \cite{Mendelson} according to Alan Slomson (personal communication).} (and repeated at the start of Section 3.1 above). Consider a vocabulary  $\tau= \{P\}$  where $P$ is a unary predicate. Then every instance of an axiom schema from the system in \emph{\cite{Bell&Slomson}} is  $true^*$ in $\mathfrak{A}$ and the rules  preserve this property whereas the following formula is not: $$\exists x P (x) \rightarrow \exists x \neg \neg P (x).$$ 

\end{Pro}
    
    \begin{Rmk} \emph{The inquisitive reader may then ask where exactly the mistake in the purported completeness proof  in \cite[Thm. 3.5.1]{Bell&Slomson} is. It is in the final step, where they claim that the equivalence class of formulas provably equivalent to $\exists v_n \psi(v_0/x_0, \dots, v_{n-1}/x_{n-1}, v_n)$ coincides with that of $\neg \forall v_n \neg \psi(v_0/x_0, \dots, v_{n-1}/x_{n-1}, v_n)$. Here they would have needed the principle $\exists x \psi \leftrightarrow \neg \forall x  \neg \psi$, but since the addition of this schema would imply the provability of $\exists x \psi \rightarrow \exists x \neg \neg \psi$, it cannot be a theorem of the system in \cite{Bell&Slomson}.}
  \hspace*\fill$\blacktriangleleft$  
    \end{Rmk}

    \section{Conclusion}
    
    This article draws attention to a relatively subtle point: completeness of proof systems – here illustrated by axiomatic or `Hilbert' systems – is very sensitive to  the choice of \emph{all} logical primitives, not only the propositional connectives. This  did not appear to be sufficiently well known, as witnessed by the error in \cite{Bell&Slomson}. As we mentioned before, H. Hi{\. z} \cite{Hiz} had already  warned that a  `translation of a complete set of axioms to another set of primitives would be complete only if from the resulting axioms the definitions of the first set of primitives followed'.  In this paper we have shown exactly why things go wrong when the advice is not followed at the level of the quantifiers: we can lose replacement of equivalents in our logic.\footnote{The need for checking against the loss (under change of primitives) of this replacement property, known in modal logic as congruentiality, was stressed already in \cite{Makinson}.} Observe that, at the propositional level, already both Hi{\. z}'s counterexample in  \cite{Hiz}  (that is, $\neg(p\wedge \neg p)$) and Shapiro's in \cite{Shapiro} (namely $(p\wedge p) \rightarrow p$) can be trivially interpreted as displaying a failure of replacement of equivalents.

    \bibliographystyle{acm}
\bibliography{BCH.bib}

\end{document}